\documentclass[final, sort&compress]{elsarticle} 
\usepackage[text={14.7cm, 22.5cm}, centering] {geometry} 

\usepackage{amsmath,amsthm,amsfonts,amssymb,cases}
\usepackage{enumerate, paralist}
\usepackage{longtable,tabularx,booktabs}
\usepackage{tikz} 

\usepackage{xcolor}
\usepackage[colorlinks,allcolors=teal]{hyperref} 

\newcommand{\F}{\mathbb{F}}

\newtheorem{thm}{Theorem}
\newtheorem{cor}{Corollary}
\newtheorem{lem}{Lemma}

\begin{document}
\begin{frontmatter}

\title{
On inverses of some permutation polynomials over finite \\fields of characteristic three
\tnoteref{foot}}
\tnotetext[foot]{
This paper has been published by ELSVIER available at
\href{https://doi.org/10.1016/j.ffa.2020.101670}
{https://doi.org/10.1016/j.ffa.2020.101670}.
Please refer to this paper as: \color{purple}{
Y. Zheng, F. Wang, L. Wang, W. Wei. On inverses of some permutation polynomials over finite fields of characteristic three. Finite Fields and Their Applications, 66:101670, 2020.
}}

\author[QF,DG,GL]{Yanbin Zheng}

\author[GL]{Fu Wang}

\author[JN]{Libo Wang\corref{cor}}%

\author[DG]{Wenhong Wei}

\address{\textnormal{Corresponding author: Libo Wang, Email: wanglibo12b@mails.ucas.edu.cn}}
\address[QF]{School of Mathematical Sciences, Qufu Normal University, Qufu, China}
\address[DG]{School of Computer Science and Technology,
         Dongguan University of Technology, Dongguan, China}
\address[GL]{Guangxi Key Laboratory of Cryptography and Information Security, Guilin University of Electronic Technology, Guilin, China}
\address[JN]{College of Information Science and Technology, Jinan University, Guangzhou, China}

\begin{abstract}
  By using the piecewise method, Lagrange interpolation formula and Lucas' theorem, we determine explicit expressions of the inverses of a class of reversed Dickson permutation polynomials and some classes of generalized cyclotomic mapping permutation polynomials over finite fields of characteristic three.
\end{abstract}

\end{frontmatter}

\section{Introduction}
For $q$ a prime power, let $\F_{q}$ denote the finite field with $q$ elements,
$\F_{q}^{*} =\F_{q}\setminus\{0\}$, and $\F_{q}[x]$ the ring of polynomials over~$\F_{q}$.
A polynomial $f(x) \in \F_{q}[x]$ is called a permutation polynomial (PP) of~$\F_{q}$
if it induces a bijection from $\F_{q}$ to itself. For any PP $f(x)$ of~$\F_{q}$,
there exists a polynomial $f^{-1}(x) \in \F_{q}[x]$ such that $f^{-1}(f(c)) =c$
for each $c \in \F_{q}$ or equivalently $f^{-1}(f(x)) \equiv x \pmod{x^{q} -x}$,
and the polynomial $f^{-1}(x)$ is unique in the sense of reduction modulo $x^q -x$.
Hence $f^{-1}(x)$ is defined as the composition inverse of $f(x)$,
and we simply call it the inverse of $f(x)$ on $\F_{q}$.

Recently, some classes of PPs are found;
see for example~\cite{NLi17-tri,DWu17,Zha17-tri,TuZ18,XuCP18}
for PPs of the form $x^rh(x^{q-1})$ of $\F_{q^2}$,
\cite{LWang18,DZheng19} for PPs of the form $(x^{q} -x +c)^s +L(x)$ of $\F_{q^2}$,
\cite{LLi18,Zheng-DCC,XuFZ19} for PPs of the form
$(ax^{q} +bx +c)^r\phi((ax^{q} +bx +c)^s) +ux^{q}+vx$ of $\F_{q^2}$,
\cite{LiQSL19} for PPs with low boomerang uniformity,
and \cite{Bartoli18,Hou18,ChouH19} for PPs studied
using the Hasse-Weil bound and Hermite's criterion.
For a detailed introduction to the developments on PPs,
we refer the reader to~\cite{Hou15,WangIndex19} and the references therein.

The construction of PPs of finite fields is not an easy subject.
However, the problem of determining the inverse of a PP seems to be
an even more complicated problem. In fact, there are a few known classes of
PPs whose inverses have been obtained explicitly;
see for example~\cite{KLi18-1, MR-1, Wang-1,Zheng-2}
for PPs of the form $x^rh(x^{(q-1)/d})$,
\cite{Wu-L-bi, Wu-L} for linearized PPs,
\cite{Wang-cyc2,Zheng-2} for generalized cyclotomic mapping PPs,
\cite{Zheng-1} for general piecewise PPs,
\cite{Involutions, NiuLQW19, ZhengYL+2019}
for involutions over finite fields,
and~\cite{TW-1,TW17} for more general classes of PPs.
For a brief summary of the results concerning the inverses of PPs,
we refer the reader to~\cite{Zheng-3} and the references therein.

The Dickson polynomial $D_{n}(x,a)$ of degree $n$
in the indeterminate $x$ and with parameter $a \in \F_{q}$ is given as
\[
D_{n}(x,a) = \sum_{i=0}^{\lfloor n/2\rfloor}
             \frac{n}{n-i} \binom{n-i}{i}(-a)^i x^{n-2i},
\]
where $\lfloor n/2\rfloor$ denotes the largest integer $\leq n/2$,
and the term $\frac{n}{n-i} \binom{n-i}{i}$ is an integer.
By reversing the roles of the indeterminate $x$ and the parameter $a$ in $D_{n}(x,a)$,
the $n$-th reversed Dickson polynomials $D_{n}(a,x)$ was defined in~\cite{Hou-RDP09} by
\begin{equation*}
D_{n}(a, x) = \sum_{i=0}^{\lfloor n/2\rfloor}\frac{n}{n-i}\binom{n-i}{i}(-1)^{i}a^{n-2i}x^{i},
\end{equation*}
and their permutation properties were studied in~\cite{Hou-RDP09}.
Several families of reversed Dickson PPs over finite fields were given in~\cite{Hou-RDP09, Hou11},
which covered all the reversed Dickson PPs over $\F_{q}$ with $q < 200$.
Then, the notion of (reversed) Dickson polynomials of the $(k+1)$-th kind was introduced
in~\cite{Wang-DP12}, and the factorization and the permutation behavior of Dickson polynomials of the third kind were studied in~\cite{Wang-DP12}.
Some necessary conditions for reversed Dickson polynomials of the first and second kinds
to be PPs of finite fields were given in~\cite{Hou-RDP10} and \cite{HongQZ16}, respectively.
The permutation behavior of reversed Dickson polynomials of the $(k+1)$-th kind was further investigated in~\cite{Fernando-RDP17,Fernando-RDP18,Fernando17arXiv}.
In particular, Hou~\cite{Hou11} proved the following result.

\begin{lem}[\textup{\cite[Theorem1.1]{Hou11}}] \label{PP-3^n}
Let $n$ be a positive even integer. Then
\begin{equation*}\label{f-hou}
    f(x)=(x -x^2 -x^3)x^{\frac{3^{n}-1}{2}} -x +x^2
\end{equation*}
is a PP of $\F_{3^n}$.
\end{lem}
Since the reversed Dickson polynomial $D_{3^{n}+5}(1,x) = f(1-x)-1$,
Hou equivalently proved that $D_{3^{n}+5}(1,x)$ is a PP of $\F_{3^n}$ for even $n$.

The purpose of this paper is to find the inverse of $f(x)$ in Lemma~\ref{PP-3^n}. The main idea is the combination of the piecewise method in~\cite{Zheng-2} and some techniques in~\cite{Zheng-3}.

The rest of the paper is organized as follows.
Section~2 gives a formula for the inverse of a class of piecewise PPs $f(x)$,
which converts the problem of determining the inverse of $f(x)$ on $\F_{q}$
to the problem of computing the inverse $f_{s}^{-1}(x)$ of piece function $f_{s}(x)$
when restricted to a subset $C_{s}$ of $\F_{q}$ for all $s$.
Then an expression of $f_{s}^{-1}(x)$ is presented in Theorem~\ref{f-1-C01},
which provides all the coefficients of $f_{s}^{-1}(x)$ by
computing the coefficients of $x^{(q-3)/2}$ and $x^{q-2}$
in $f_{s}(x)^{k} \pmod{x^{q}-x}$ for $1 \leq k \leq (q-3)/2$.
By applying the results in Section~2 to $f(x)$ in Lemma~\ref{PP-3^n}, the coefficients of $f^{-1}(x)$ are reduced into two classes of binomial coefficients in Section~3.
Section~4 gives explicit values of these binomial coefficients
by using a congruence of binomial coefficients and Lucas' theorem.
In other words, we determine the inverse of $f(x)$ as follows.
\begin{thm}\label{inv-hou}
The inverse of $f(x)$ in \textup{Lemma~\ref{PP-3^n}} on $\F_{3^n}$ is
\[
 f^{-1}(x)
= x^{3^{n-1}}\big(x^{\frac{3^{n}-1}{2}} +1\big)
+ \Bigg(\sum_{j=0}^{n-1}\sum_{k=0}^{n-1}
    (-1)^{j+k} x^{\frac{3^j+3^k}{2}} \Bigg)
    \big(x^{\frac{3^{n}-1}{2}} -1\big).
\]
Furthermore, the inverse of $D_{3^{n}+5}(1,x)$ on $\F_{3^n}$ is
$D_{3^{n}+5}^{-1}(1,x)= 1-f^{-1}(x+1)$.
\end{thm}

In the last section, by an argument similar to the one used
in Theorem~\ref{inv-hou}, we also obtain explicit inverses of
some generalized cyclotomic mapping PPs studied in~\cite{Wang-cyc}.

\section{The inverse of a class of piecewise PPs}

The piecewise methods for constructing PPs and their inverses were summarized in \cite{FH-pw,CHZ14} and~\cite{Zheng-2} respectively.
Applying these methods, we can easily get the following~result.

\begin{lem}\label{inv-twop}
Let $q$ be odd and $f_{0}(x), f_{1}(x) \in \F_{q}[x]$. Define
\[
  f(x)
  = \tfrac{1}{2}f_{0}(x)(1+x^{\frac{q-1}{2}})
   +\tfrac{1}{2}f_{1}(x)(1-x^{\frac{q-1}{2}})
\quad \text{with $f(0) =0$,}
\]
$C_{0} = \{e^2 : e \in \F_{q}^{*}\}$, and $C_{1} = \F_{q}^{*} \setminus C_{0}$.
Then $f(x)$ is a PP of $\F_q$ if and only if $f_{s}$ is injective on $C_{s}$
and $0 \not\in f_{s}(C_{s})$ for $s \in \{0, 1\}$,
and $f_{0}(C_{0}) \cap f_{1}(C_{1}) = \emptyset$.
Assume $f(x)$ is a PP of $\F_q$, and $f_{s}^{-1}(x) \in \F_q[x]$ satisfies that
$f_{s}^{-1}(0) =0$ and $f_{s}^{-1} (f_{s}(c))=c$ for any $c \in C_{s}$ and
$s \in \{0, 1\}$.
\begin{enumerate}[\upshape(i)]
  \item If $f_{s}$ maps $C_{s}$ into $C_{s}$ for $s \in \{0, 1\}$,
        then the inverse of $f(x)$ on $\F_{q}$ is
        \begin{equation}\label{f-1-s}
            f^{-1}(x)
            = \tfrac{1}{2}f_{0}^{-1}(x)(1 + x^{\frac{q-1}{2}})
            +\tfrac{1}{2}f_{1}^{-1}(x)(1 - x^{\frac{q-1}{2}}).
        \end{equation}
  \item If $f_{s}$ maps $C_{s}$ into $C_{t}$ for $s \neq t \in \{0, 1\}$,
        then the inverse of $f(x)$ on $\F_{q}$ is
        \begin{equation*}
            f^{-1}(x)
            = \tfrac{1}{2}f_{0}^{-1}(x)(1 - x^{\frac{q-1}{2}})
            +\tfrac{1}{2}f_{1}^{-1}(x)(1 + x^{\frac{q-1}{2}}).
        \end{equation*}
\end{enumerate}
\end{lem}
Lemma~\ref{inv-twop} converts the problem of  determining $f^{-1}(x)$
into the problem of computing $f_{s}^{-1}(x)$, the inverse of piece function
$f_{s}(x)$ when restricted to $C_{s}$. In Lemma~\ref{inv-twop},
if $q=3$, then $f(x) = f^{-1}(x)=\pm x$ in the sense of reduction modulo $x^{3}-x$.
We will give an expression of $f_{s}^{-1}(x)$ for $q>3$ after the following lemma.

\begin{lem}\label{sum}
For $q$ an odd prime power, let $C_0=\{e^2: e\in\F_{q}^{*}\}$
and $C_1=\F_{q}^{*} \setminus C_0$. Then
\[
\sum_{a \in C_{s}}a^k =
\begin{cases}
  \quad -2^{-1}  & \text{if ~$k = q-1$,} \\
  (-1)^{s+1} 2^{-1} & \text{if ~$k = (q-1)/2$,} \\
  \qquad 0      & \text{if $1 \leq k \leq q-2$ and $k \neq (q-1)/2$.}
\end{cases}
\]
\end{lem}
\begin{proof}
Let $A=\sum_{a \in C_1}a^k$ and $\xi$ a primitive element of $\F_{q}$. Obviously,
\begin{equation}\label{sum-all}
(\xi^k +1)A
 = \sum_{a \in C_0}a^k + \sum_{a \in C_1}a^k
 = \sum_{a \in \F_{q}}a^k
 = \begin{cases}
  -1 & \text{if $k = q-1$,}\\
  ~~0  & \text{if $k = 1$, $2$, $\ldots$, $q-2$.}
  \end{cases}
\end{equation}
If $k = q-1$, then $\xi^k=1$ and $2A=-1$, so $A=-1/2$.
If $k = (q-1)/2 $, then $a^k = -1$ for $a \in C_1$.
Hence $A=(-1)(q-1)/2=1/2$. If $1 \leq k \leq q-2$ and $k \neq (q-1)/2$,
then $\xi^k \neq -1$ and $(\xi^k +1)A=0$. Thus $A=0$.
Then the result follows from~$\eqref{sum-all}$.
\end{proof}

\begin{thm}\label{f-1-C01}
For an odd prime power $q > 3$, let $C_0=\{e^2: e\in\F_{q}^{*}\}$
and $C_1=\F_{q}^{*} \setminus C_0$.
For $s, t \in \{0,1\}$, assume $f_{s}(x) \in \F_q[x]$
induces a bijection from $C_{s}$ to $C_{t}$, and
\begin{equation}\label{h-q-1-i}
   f_{s}(x)^{q-1-i} \equiv \sum_{0 \leq k \leq q-1} b_{ik} x^k \pmod{x^q-x},
\end{equation}
where $(q+1)/2 \leq i \leq q-2$.
Then the inverse of $f_{s}(x)$ when restricted to $C_{s}$ is
\begin{equation}\label{Rf2-1}
f_{s}^{-1}(x) = \sum_{(q+1)/2 \leq i \leq q-2}
   (-1)^{s+t}\big(b_{i,(q-3)/2} + (-1)^{s}b_{i,q-2} \big) x^{i-\frac{q-1}{2}}
\end{equation}
 in the sense of reduction modulo $x^{\frac{q-1}{2}} +(-1)^{t+1}$.
\end{thm}
\begin{proof}
Let $f_{s}^{-1}(x)=\sum_{i=0}^{q-1} c_i x^i \in \F_q[x]$.
Then by the Lagrange interpolation formula,
\[\begin{split}
f_{s}^{-1}(x)
& =\sum_{a \in C_{s}} a \big(1-(x-f_{s}(a))^{q-1}\big) \\
& = \sum_{a \in C_{s}} a \Big(-\sum_{1 \leq i \leq q-1}(-1)^i(-f_{s}(a))^{q-1-i}x^i\Big) \\
& = \sum_{1 \leq i \leq q-1}\Big(-\sum_{a \in C_{s}} a f_{s}(a)^{q-1-i}\Big)x^i.
\end{split}\]
Hence $c_0=0$ and
\begin{equation}\label{bi=ef2}
  c_{i} = -\sum_{a \in C_{s}} a f_{s}(a)^{q-1-i},
   \quad \text{where $1 \leq i \leq q-1$.}
\end{equation}
Then by~\eqref{h-q-1-i} and Lemma~\ref{sum},
\begin{equation}\label{bi-aik}
\begin{split}
 c_{i} & = -\sum_{a \in C_{s}} a\sum_{0 \leq k \leq q-1} b_{ik}a^k
        = -\sum_{0 \leq k \leq q-1} b_{ik}\sum_{a \in C_{s}} a^{k+1}  \\
       & = 2^{-1}(b_{i,q-2} +(-1)^{s} b_{i,(q-3)/2}).
\end{split}
\end{equation}

We next reduce the degree of $f_{s}^{-1}(x)$.
Since $f_{s}(x)$ induces a bijection from $C_{s}$ to $C_{t}$, we have
$f_{s}(a) \in C_{t}$  and $f_{s}(a)^{(q-1)/2}=(-1)^{t}$ for any $a \in C_{s}$, and so
\[
      f_{s}(a)^{q-1-i}
  = (-1)^{t} f_{s}(a)^{q-1-i}/f_{s}(a)^{\frac{q-1}{2}}
  = (-1)^{t} f_{s}(a)^{q-1-(i+\frac{q-1}{2})}.
\]
Substituting it into \eqref{bi=ef2} yields
\begin{equation}\label{bi=-bi+s}
    c_{i}  = (-1)^{t} c_{i+\frac{q-1}{2}},
    \quad 1 \leq i \leq (q-1)/2.
\end{equation}
In particular, if $q > 3$, then
\begin{equation}\label{bs=0}
   c_{(q-1)/2} = (-1)^{t} c_{q-1} =(-1)^{t+1}\sum_{a \in C_{s}} a =0,
\end{equation}
where the last identity follows from Lemma~\ref{sum}. Therefore,
\begin{equation}\label{fs-1}
\begin{split}
  f_{s}^{-1}(x)
  & \overset{\eqref{bs=0}}{=}
     \sum_{1 \leq i \leq (q-3)/2}
     \big(c_{i}x^{i} +c_{i+\frac{q-1}{2}}x^{i+\frac{q-1}{2}} \big) \\
  & \equiv \sum_{1 \leq i \leq (q-3)/2} \big( c_i +(-1)^{t} c_{i+\frac{q-1}{2}} \big)x^{i}
    \pmod{x^{\frac{q-1}{2}} +(-1)^{t+1}}\\
  & \overset{\eqref{bi=-bi+s}}{=}
    \sum_{1 \leq i \leq (q-3)/2} 2(-1)^{t} c_{i+\frac{q-1}{2}} x^{i}  \\
  & = \sum_{(q+1)/2 \leq i \leq q-2} 2(-1)^{t}c_{i} x^{i-\frac{q-1}{2}}.
\end{split}
\end{equation}
Substituting \eqref{bi-aik} into \eqref{fs-1} gives the desired result.
\end{proof}

\section{The inverse of the PP in Lemma~\ref{PP-3^n}}

In this section, we will employ the results in Section~2
to compute the inverse of the PP $f(x)$ in Lemma~\ref{PP-3^n}.
First, let $n>0$ be even, $C_0=\{e^2: e \in\F_{3^{n}}^{*}\}$
and $C_1=\F_{3^{n}}^{*} \setminus C_0$. Then $x^{(3^{n}-1)/2}=1$ if $x \in C_0$, and $x^{(3^{n}-1)/2}= -1$ if $x \in C_1$. Therefore,
\begin{equation*}
f(x)=(x -x^2 -x^3)x^{\frac{3^{n}-1}{2}} -x +x^2
\end{equation*}
can be written as
\[
 f(x) =
 \begin{cases}
   0              & \text{if $x = 0$,} \\
   f_{0}(x) := - x^3          & \text{if $x \in C_0$,} \\
   f_{1}(x) := x(x+1)^2  & \text{if $x \in C_1$.}
 \end{cases}
\]
Lemma~\ref{PP-3^n} stated that $f(x)$ is a PP of~$\F_{3^{n}}$.
It means that $f_{0}(x)$ (resp. $f_{1}(x)$) induces an injection on $C_{0}$ (resp. $C_{1}$), and $f_{0}(C_{0}) \cap f_{1}(C_{1}) = \emptyset$.
Since $n$ is even, $3^{n} \equiv 1 \pmod{4}$, and so $-1 \in C_{0}$.
Hence $f_0(x)$ (resp. $f_1(x)$) induces a permutation on $C_{0}$ (resp. $C_{1}$).

Since $(x^3)^{3^{n-1}} = x^{3^{n}} = x$ for any $x \in C_0$,
the inverse of $f_{0}(x)$ on $C_{0}$ is
\begin{equation}\label{f0-1}
f_{0}^{-1}(x)= -x^{3^{n-1}}.
\end{equation}
We next apply Theorem~\ref{f-1-C01} to determine the inverse of $f_{1}(x)$ on $C_{1}$.
Denote $q = 3^{n}$, $u=(q-3)/2$, $v=q-2$ and $w_{i}=q-1-i$, where $(q+1)/2 \leq i \leq q-2$. Then
\begin{equation}\label{f2^wi}
  f_{1}(x)^{w_{i}}
  = x^{w_{i}}(x+1)^{2w_{i}}
  = \sum_{0 \leq j \leq 2w_{i}} \binom{2w_{i}}{j} x^{w_{i}+j}.
\end{equation}
The degree of $f_{1}(x)^{w_{i}}$ is $3w_{i}$,
and $3 \leq 3w_{i} < q-1 +u$, as shown in Figure~\ref{Range-3wi}.
\begin{figure}[ht]
  \centering
\begin{tikzpicture}
\filldraw
 (-0.7,0) --
 (0,0)     circle (1pt) node[above] {$3$} --
 (1.2,0)   circle (1pt) node[above] {$u$} --
 (2.4,0)   circle (1pt) node[above] {$v$} --
 (3.8,0)   circle (1pt) node[above] {$q\!-\!1\!+\!u$} --
 (5.5,0)   circle (1pt) node[above] {$q\!-\!1\!+\!v$};
\draw[->] (5.5,0) -- (6.5,0);
\draw  (0,0) -- (0,-0.6);
\draw  (3.6,0) -- (3.6,-0.6);
\draw[<-] (0,-0.4) -- (1.55,-0.4);
\draw[->] (2.25,-0.4) -- (3.6,-0.4);
\draw (1.9, -0.4) node{$3w_{i}$};
\end{tikzpicture}
\caption{The range of $3w_{i}$}\label{Range-3wi}
\end{figure}
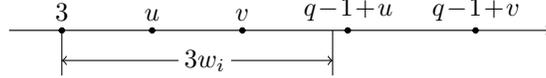
Hence the coefficient of $x^u$ in~\eqref{h-q-1-i} equals
the coefficients of $x^u$ in~\eqref{f2^wi}, i.e., $b_{iu}=\binom{2w_{i}}{u-w_{i}}$.
Similarly, $b_{iv}=\binom{2w_{i}}{v-w_{i}}$.
If $3w_{i} < u$, i.e., $i > (5q-3)/6$, then $b_{iu} = 0$.
If $3w_{i} < v$, i.e., $i \geq 2q/3$, then $b_{iv}=0$.
Hence we only need to consider the binomial coefficients
\[
\begin{split}
 b_{iu} & = \binom{2q-2-2i}{i-\frac{q+1}{2}}
    \quad \text{for $(q+1)/2 \leq i \leq (5q-3)/6$,}  \\
 b_{iv} & = \binom{2q-2-2i}{i-1} \quad
    \text{for $(q+1)/2 \leq i < 2q/3$.}
\end{split}
\]
\begin{lem}[{\cite[Lemma~9]{Zheng-3}}]\label{lem-bico-qm}
Let $q$ be a prime $p$ power, and let $m$, $k$ be integers with $0 \leq k < q$. Then
\[
\binom {q+m}{k} \equiv \binom {m}{k} \pmod{p},
\]
where $\binom {m}{k}=m(m-1)\cdots(m-k+1)/k!$.
\end{lem}

Employing Lemma~\ref{lem-bico-qm} and the fact
$\binom {-m}{k} = (-1)^k \binom {m+k-1}{k}$ for $m$, $k > 0$, we have
\begin{equation}\label{biv}
 b_{iv} \equiv \binom{-2-2i}{i-1} \equiv (-1)^{i-1}\binom {3i}{i-1} \pmod{3},
\end{equation}
where $(q+1)/2 \leq i < 2q/3$.
Similarly, if $(q+1)/2 \leq i \leq (5q-3)/6$, then
\begin{equation}\label{biu}
 b_{iu} \equiv \binom{-2-2i}{i-\frac{q+1}{2}}
 \equiv (-1)^{i-\frac{q+1}{2}}\binom{3i-\frac{q-1}{2}}{i-\frac{q+1}{2}} \pmod{3}.
\end{equation}

By Lemma~\ref{inv-twop}, Theorem~\ref{f-1-C01}, \eqref{f0-1}, \eqref{biv} and \eqref{biu},
the inverse of $f(x)$ in Lemma~\ref{PP-3^n}~on~$\F_{3^n}$~is
\begin{equation}\label{f-1-abs}
 f^{-1}(x)
= x^{3^{n-1}}\big(x^{\frac{3^{n}-1}{2}} +1\big)
+ f_{1}^{-1}(x)\big(x^{\frac{3^{n}-1}{2}} -1\big),
\end{equation}
where
\begin{equation}\label{f1-1-bi-co}
\begin{split}
f_{1}^{-1}(x) & =
   \sum_{(3^{n}+1)/2 \leq i \leq (5\cdot3^{n}-3)/6} (-1)^{i-\frac{3^{n}+1}{2}}
   \binom{3i-\frac{3^{n}-1}{2}}{i-\frac{3^{n}+1}{2}} x^{i-\frac{3^{n}-1}{2}} \\
  & \quad + \sum_{(3^{n}+1)/2 \leq i \leq 2\cdot3^{n}/3}
  (-1)^{i} \binom{3i}{i-1} x^{i-\frac{3^{n}-1}{2}}
\end{split} \end{equation}
in the sense of reduction modulo $x^{\frac{3^{n}-1}{2}} +1$.

\section{Explicit values of binomial coefficients}

This section will give the explicit values
of binomial coefficients in~\eqref{f1-1-bi-co}.

\begin{thm}
  Let $1 \leq i < 3^{n}$ with $n \geq 1$.
  Then {\large$\binom{3i}{i-1}$} $\not\equiv 0 \pmod{3}$ if and only if
  $i=(3^{k}-1)/2$, where $k=1$, $2$, $\ldots$, $n$.
\end{thm}

\begin{proof}
If $1 \leq i < 3^{n}$, then we can write
$
i = i_0 +i_1 3 +i_2 3^2+ \cdots +i_{n-1} 3^{n-1},
$
where $i_t=0$, $1$, $2$ for $0 \leq t \leq n-1$. Then
\[
 \begin{split}
  3i   & = ~~ 0 ~~~    +i_0 3 +i_1 3^2 + \cdots +i_{n-2} 3^{n-1}+i_{n-1} 3^{n}, \\
  i-1  & =  i_0\!-\!1  +i_1 3 +i_2 3^2 + \cdots +i_{n-1} 3^{n-1}.
  \end{split}
\]
By Lucas' theorem, if $i_0 = 0$, then
$\binom{3i}{i-1} \equiv \binom{0}{2} \equiv 0 \pmod{3}$,
and if $i_0 = 1, 2$, then
\[
\binom{3i}{i-1} \equiv \binom{0}{i_0\!-\!1}\binom{i_0}{i_1}
\binom{i_1}{i_2}\cdots\binom{i_{n-2}}{i_{n-1}} \binom{i_{n-1}}{0}  \pmod{3}.
\]
Hence $\binom{3i}{i-1} \not\equiv 0 \pmod{3}$
if and only if one of the following conditions holds:
\begin{enumerate}
  \item $i_0=1$ and $1 \geq i_1 \geq i_2 \geq \cdots
        \geq i_{n-2} \geq i_{n-1} \geq 0$;
  \item $i=1$, $1+3$, $1+3+3^2$, $\ldots$, $1+3+3^2+\cdots +3^{n-1}$;
  \item $i=(3^{k}-1)/2$, where $k=1$, $2$, $\ldots$, $n$. \qedhere
\end{enumerate}
\end{proof}
\begin{cor}\label{3i_i-1=0}
  Let $(3^{n}-1)/2 < i < 3^{n}$ with $n \geq 1$.
  Then {\large$\binom{3i}{i-1}$} $\equiv 0 \pmod{3}$.
\end{cor}

It is easy to obtain the following $3$-adic expansion:
\[
(5\cdot3^{n}-3)/6 -(3^{n}+1)/2 = 2+2\cdot3+2\cdot3^2+\cdots+2\cdot3^{n-2}.
\]
Thus, when $(3^{n}+1)/2 \leq i \leq (5\cdot3^{n}-3)/6$, we can write
\begin{equation}\label{i-(q+1)/2}
i = (3^{n}+1)/2 +i_0 +i_1\cdot3 +i_2\cdot3^2+\cdots +i_{n-2}\cdot3^{n-2},
\end{equation}
where $i_t=0$, $1$, $2$ for $0 \leq t \leq n-2$.
Then by an argument similar to the one used in \cite[Theorems~11 and 12]{Zheng-3},
we obtain the following two results.

\begin{thm}\label{c-ne0}
Let $(3^{n}+1)/2 \leq i \leq (5\cdot3^{n}-3)/6$ with $n \geq 1$.
Then the following statements are equivalent:
\begin{enumerate}[$(i)$]
  \item $\displaystyle\binom {3i-\frac{3^{n} -1}{2}}{i-\frac{3^{n} +1}{2}}$ $\not\equiv 0 \pmod{3}$;
  \item $2 \geq i_0 \geq i_1 \geq \cdots \geq i_{n-2} \geq 0$,
        where $i_0, i_1, \ldots, i_{n-2}$ are defined by~\eqref{i-(q+1)/2};
  \item $i = \frac{3^{n} +1}{2} + \frac{3^j -1}{2} +\frac{3^k-1}{2}$,
        where $0 \leq j \leq k\leq n-1$.
\end{enumerate}
\end{thm}

\begin{thm}\label{Aeq1}
Let $i = \frac{3^{n} +1}{2} + \frac{3^j -1}{2} +\frac{3^k-1}{2}$,
where $0 \leq j \leq k\leq n-1$. Then
\[
\binom {3i-\frac{3^{n}-1}{2}}{i-\frac{3^{n}+1}{2}} \equiv
\begin{cases}
   1 \pmod{3} & \text{if $j = k$,}\\
   2 \pmod{3} & \text{if $j < k$.}\\
\end{cases}
\]
\end{thm}
According to Theorem~\ref{Aeq1}, we obtain the following result.
\begin{thm}\label{Beq1}
Let $i = \frac{3^{n} +1}{2} + \frac{3^j -1}{2} +\frac{3^k-1}{2}$,
where $0 \leq j \leq k\leq n-1$. Then
\[
(-1)^{i-\frac{3^{n}+1}{2}}
\binom {3i-\frac{3^{n}-1}{2}}{i-\frac{3^{n}+1}{2}} \equiv
\begin{cases}
  \quad 1 ~\qquad \pmod{3} & \text{if $j = k$,}\\
   (-1)^{j+k+1} \pmod{3} & \text{if $j < k$.}\\
\end{cases}
\]
\end{thm}
\begin{proof}
Since
$3^{m}=(1+2)^{m}=1+\binom {m}{1}2 +\binom {m}{2}2^2+\cdots+\binom {m}{m}2^m$,
we have $(-1)^{\frac{3^m-1}{2}}=(-1)^{m}$, and so
$(-1)^{i-\frac{3^{n}+1}{2}} = (-1)^{\frac{3^j-1}{2} +\frac{3^k-1}{2}} = (-1)^{j+k}$.
Substituting it into Theorem~\ref{Aeq1} gives the desired result.
\end{proof}

By Corollary~\ref{3i_i-1=0}, Theorems~\ref{c-ne0} and~\ref{Beq1},
we can write~\eqref{f1-1-bi-co} as the following form:
\begin{equation}\label{f1-1-ajk}
\begin{split}
f_{1}^{-1}(x)
& = \sum_{0\leq k \leq n-1} x^{3^k}
    +\sum_{0 \leq j < k \leq  n-1} 2(-1)^{j+k} x^{\frac{3^j+3^k}{2}} \\
& = \sum_{0\leq j = k \leq n-1} (-1)^{j+k} x^{\frac{3^j+3^k}{2}}
    +\sum_{0 \leq j \neq k \leq  n-1} (-1)^{j+k} x^{\frac{3^j+3^k}{2}} \\
& = \sum_{0 \leq j, k \leq  n-1}
    (-1)^{j+k} x^{\frac{3^j+3^k}{2}}.
\end{split}
\end{equation}
Substituting~\eqref{f1-1-ajk} into~\eqref{f-1-abs}
completes the proof of Theorem~\ref{inv-hou}.
\section{Slight generalization}

In this section, we also let $C_0=\{e^2: e \in\F_{3^{n}}^{*}\}$
and $C_1=\F_{3^{n}}^{*} \setminus C_0$.
Let $\eta$ be the quadratic character.
By an argument similar to that used in Theorem~\ref{inv-hou},
we deduce the inverses of some generalized PPs studied in~\cite{Wang-cyc}.

\begin{lem}[{\cite[Theorem~4.7]{Wang-cyc}}]\label{W4.7}
Let $\alpha$, $\beta$, $\gamma$, $\theta \in \F_{3^n}^{*}$ with $n \geq 1$, and let
\begin{equation}\label{f4.7}
 f(x) =
 \begin{cases}
   0              & \text{if $x = 0$,} \\
   \alpha(x^3 +\gamma x^2 +\gamma^2 x)  & \text{if $x \in C_0$,} \\
   \beta (x^3 +\theta x^2 +\theta^2 x)  & \text{if $x \in C_1$.}
 \end{cases}
\end{equation}
Then $f(x)$ is a PP of $\F_{3^n}$ if and only if
$\eta(\gamma)=-1$, $\eta(\theta)=1$ and $\eta(\alpha) =\eta(\beta)$.
\end{lem}

Lemma~\ref{PP-3^n} is a special case of Lemma~\ref{W4.7}
for $n$ is even, $\gamma=0$ and $\alpha = -\beta = \theta = -1$.
The following result gives the inverse of $f(x)$ in Lemma~\ref{W4.7}.

\begin{thm}\label{inv-4.7}
If $f(x)$ in Lemma~\ref{W4.7} is a PP of $\F_{3^n}$
and $\eta(\alpha) =(-1)^{m}$ with $m \in \{0,1\}$, then
\[
  f^{-1}(x)
  = - u(x)\big(1+(-1)^{m}   x^{(3^n-1)/2}\big)
    - v(x)\big(1+(-1)^{m+1} x^{(3^n-1)/2}\big),
\]
where
\begin{equation}\label{u}
 u(x) =  \sum_{0 \leq j, k \leq n-1}
       \gamma\big(\alpha^{-1}\gamma^{-3} x \big)^{\frac{3^j+3^k}{2}},
\end{equation}
\begin{equation}\label{v}
 v(x)= \sum_{0 \leq j, k \leq n-1}
      \theta\big( \beta^{-1} \theta^{-3} x \big)^{\frac{3^j+3^k}{2}}.
\end{equation}
\end{thm}
\begin{proof}
Let $\tau, \lambda \in \F_{3^n}^{*}$, $s \in \{0,1\}$ and
$
f_{s}(x) = \tau(x^3 +\lambda x^2 + \lambda^2 x) = \tau x (x-\lambda)^{2}.
$
If $f(x)$ in Lemma~\ref{W4.7} is a PP of $\F_{3^n}$
and $\eta(\tau) =(-1)^{m}$ with $m \in \{0,1\}$,
then $f_{s}(x)$ induces a bijection from $C_{s}$ to $C_{t}$, where $s, t \in \{0,1\}$.
Clearly, $s=t$ if $m=0$, and $s \neq t$ if $m=1$.
Hence $s+t+m \equiv 0 \pmod{2}$.
Let $w_{i} = 3^n-1-i$, where $(3^{n}+1)/2 \leq i \leq 3^{n}-2$. Then
\begin{equation*}\label{fs^wi-gen}
  f_{s}(x)^{w_{i}}
  = (\tau x)^{w_{i}}(x-\lambda)^{2w_{i}}
  = \sum_{0 \leq j \leq 2w_{i}} \tau^{w_{i}}(-\lambda)^{2w_{i}-j}
  \binom{2w_{i}}{j} x^{w_{i}+j}.
\end{equation*}
By an argument similar to that used in the previous sections, we have
\begin{align*}
  f_{s}^{-1}(x)
  & = (-1)^{s+t} \Bigg(\sum_{0 \leq k \leq n-1}
        \tau^{\frac{3^n-1}{2}-3^{k}} \lambda^{3^n-3^{k+1}} x^{3^{k}} \\
  & \qquad  +\sum_{0 \leq j < k \leq n-1}
        2\tau^{\frac{3^n-1}{2}-\frac{3^j+3^k}{2}}
        \lambda^{3^n-3\frac{3^j+3^k}{2}} x^{\frac{3^j+3^k}{2}} \Bigg)\\
  & = (-1)^{s+t}\tau^{\frac{3^n-1}{2}}
       \sum_{0 \leq j, k \leq n-1}
      \lambda^{3^n} \big(\tau^{-1}\lambda^{-3} x \big)^{\frac{3^j+3^k}{2}}  \\
  & = (-1)^{s+t+m}
       \sum_{0 \leq j, k \leq n-1}
      \lambda \big(\tau^{-1}\lambda^{-3} x \big)^{\frac{3^j+3^k}{2}}  \\
  & = \sum_{0 \leq j, k \leq n-1}
      \lambda \big(\tau^{-1}\lambda^{-3} x \big)^{\frac{3^j+3^k}{2}}.
\end{align*}
Substituting it into Lemma~\ref{inv-twop} gives the desired result.
\end{proof}

\begin{lem}[{\cite[Theorem~4.2]{Wang-cyc}}]\label{W4.2}
Let $\alpha$, $\beta$, $\theta \in \F_{3^n}^{*}$ with $n \geq 1$, and let
\begin{equation*}\label{f4.2}
 f(x) =
 \begin{cases}
   0              & \text{if $x = 0$,} \\
   \alpha x^t          & \text{if $x \in C_0$,} \\
   \beta(x^3 +\theta x^2 +\theta^2 x)  & \text{if $x \in C_1$.}
 \end{cases}
\end{equation*}
Then $f(x)$ is a PP of $\F_{3^n}$ if and only if
$\big(t, \frac{3^n-1}{2}\big)=1$, $\eta(\theta)=1$ and $\eta(\alpha) =\eta(\beta)$.
\end{lem}

\begin{thm}\label{inv-4.2}
If $f(x)$ in Lemma~\ref{W4.2} is a PP of $\F_{3^n}$
and $\eta(\alpha) =(-1)^{m}$ with $m \in \{0,1\}$, then
\[
  f^{-1}(x)
  = - u(x)\big(1+(-1)^{m}   x^{(3^n-1)/2}\big)
    - v(x)\big(1+(-1)^{m+1} x^{(3^n-1)/2}\big),
\]
where $v(x)$ is defined by~\eqref{v}, $u(x) = (\alpha^{-1}x)^s$,
and $s$ is the inverse of $t$ modulo $(3^n-1)/2$.
\end{thm}

\begin{lem}[{\cite[Corollary~4.5]{Wang-cyc}}] \label{W4.5}
Let $n$, $t$ be positive integers.
Let $\alpha$, $\beta$, $\gamma\in \F_{3^n}^{*}$ and
\begin{equation*}\label{f4.5}
 f(x) =
 \begin{cases}
   0                                    & \text{if $x = 0$,} \\
   \alpha(x^3 +\gamma x^2 +\gamma^2 x)  & \text{if $x \in C_0$,} \\
   \beta x^t                            & \text{if $x \in C_1$.}
 \end{cases}
\end{equation*}
Then $f(x)$  is a PP of $\F_{3^n}$ if and only if
$\big(t, \frac{3^n-1}{2}\big)=1$, $\eta(\gamma)=-1$ and $\eta(\alpha) =\eta(\beta)(-1)^{t+1}$.
\end{lem}

\begin{thm}\label{inv-4.5}
If $f(x)$ in Lemma~\ref{W4.5} is a PP of $\F_{3^n}$ and
$\eta(\alpha) =(-1)^{m}$ with $m \in \{0,1\}$, then
\[
  f^{-1}(x)
  = - u(x)\big(1+(-1)^{m}   x^{(3^n-1)/2}\big)
    - v(x)\big(1+(-1)^{m+1} x^{(3^n-1)/2}\big),
\]
where $u(x)$ is defined by~\eqref{u}, $v(x) = (-1)^r(\beta^{-1}x)^s$,
$r$ and $s$ are integers such that $1 \leq s < (3^n-1)/2$ and $st +r(3^n-1)/2 =1$.
\end{thm}

When $n=1$, it is easy to verify that
$f(x) \equiv f^{-1}(x) \equiv \pm x \pmod{x^{3}-x}$
in all results of this section.
Hence Theorems~\ref{inv-4.7}, \ref{inv-4.2} and \ref{inv-4.5} are all true for $n=1$.
If $\gamma =\theta =0$, then Lemmas~\ref{W4.7}, \ref{W4.2} and~\ref{W4.5}
are the special cases of \cite[Corollary~2.3]{Wang-cyc},
and their inverses are given in~\cite{Wang-cyc2,Zheng-2}.

\section*{Acknowledgments}
We are grateful to the referees for many useful comments and suggestions.


\end{document}